\pdfoutput=1
\documentclass[a4paper, reqno, final]{dmtcs-episciences}

\usepackage[utf8]{inputenc}
\usepackage[T1]{fontenc}

\usepackage{tikz}
\usetikzlibrary{arrows,backgrounds,calc}
\usepackage{amsthm, amsmath, amsfonts, amssymb}
\usepackage{enumerate}
\usepackage{mathtools}
\usepackage{pgfplots}
\pgfplotsset{compat=1.5}

\usepackage{hyperref}

\allowdisplaybreaks

\newtheorem{theorem}{Theorem}
\newtheorem{proposition}{Proposition}[section]

\newtheorem{corollary}[proposition]{Corollary}

\theoremstyle{remark}
\newtheorem*{remark}{Remark}

\theoremstyle{definition}
\newtheorem*{definition}{Definition}

\newif\ifdetails
\detailstrue
\newcommand{\TODO}[1]%
{\par\fbox{\begin{minipage}{0.9\linewidth}\textbf{TODO:} #1\end{minipage}}\par}

\newcommand{\details}[1]%
{\ifdetails\par\fbox{\begin{minipage}{0.9\linewidth}\textit{Detail:}
      #1\end{minipage}}\par\fi}

\newcommand{\Z}[0]{\mathbb{Z}}
\renewcommand{\P}[0]{\mathbb{P}}
\newcommand{\E}[0]{\mathbb{E}}
\newcommand{\V}[0]{\mathbb{V}}

\DeclarePairedDelimiter{\abs}{\lvert}{\rvert}

\newcommand{\innernode}{\tikz{\node[draw, circle, inner sep=2.5pt] {};}}

\title[Growing and Destroying Catalan--Stanley Trees]{Growing and Destroying Catalan--Stanley Trees}

\author{Benjamin Hackl\affiliationmark{1} \and Helmut Prodinger\affiliationmark{2}
\footnotetext{B.~Hackl is supported by the Austrian
  Science Fund (FWF): P~24644-N26 and by the Karl Popper Kolleg
  ``Modeling-Simulation-Optimization'' funded by the Alpen-Adria-Universit\"at Klagenfurt
  and by the Carinthian Economic Promotion Fund (KWF). This paper has been written while
  he was a visitor at Stellenbosch University.}
\footnotetext{H.~Prodinger is supported by an incentive grant of the
  National Research Foundation of South Africa.}
\footnotetext{Email-addresses:
  \href{mailto:benjamin.hackl@aau.at}{benjamin.hackl@aau.at},
  \href{mailto:hproding@sun.ac.za}{hproding@sun.ac.za}}}
\affiliation{
  Institut f\"ur Mathematik, Alpen-Adria-Universit\"at Klagenfurt, Austria\\
  Department of Mathematical Sciences, Stellenbosch University, South Africa
}

\received{2017-9-28}
\revised{2018-1-31}
\accepted{2018-2-14}

\begin{document}
\publicationdetails{20}{2018}{1}{11}{3964}
\maketitle

\begin{abstract}
  Stanley lists the class of Dyck paths where all returns to the axis are of odd length as
  one of the many objects enumerated by (shifted) Catalan numbers. By the standard
  bijection in this context, these special Dyck paths correspond to a class of rooted
  plane trees, so-called Catalan--Stanley trees.

  This paper investigates a deterministic growth procedure for these trees by which any
  Catalan--Stanley tree can be grown from the tree of size one after some number of rounds;
  a parameter that will be referred to as the age of the tree. Asymptotic analyses are
  carried out for the age of a random Catalan--Stanley tree of given size as well as for
  the ``speed'' of the growth process by comparing the size of a given tree to the size of
  its ancestors.
\end{abstract}

\section{Introduction}

It is well-known that the $n$th Catalan number $C_{n} = \frac{1}{n+1} \binom{2n}{n}$
enumerates Dyck paths of length $2n$. In~\cite{Stanley:2015:catalan}, Stanley lists a
variety of other combinatorial interpretations of the Catalan numbers, one of them being
the number of Dyck paths from $(0,0)$ to $(2n+2,0)$ such that any maximal sequence of
consecutive $(1, -1)$ steps ending on the $x$-axis has odd length. At this point it is
interesting to note that there are more subclasses of Dyck paths, also enumerated by
Catalan numbers, that are defined via parity restrictions on the length of the returns to
the $x$-axis as well (see, e.g., \cite{Callan:2005:136th-catalan}). The height of the class of Dyck paths
with odd-length returns to the origin has already been studied
in~\cite{Prodinger:2012:dyck-path-parity} with the result that the main term of the height
is equal to the main term of the height of general Dyck paths as investigated in~\cite{Bruijn-Knuth-Rice:1972}.

\begin{figure}[ht]
  \centering
  \begin{tikzpicture}[thick, scale=0.5, baseline={([yshift=-0.3em]current bounding
      box.center)}]
    \draw[help lines] (0,0) grid (20.5,3.5);
    \draw[->] (-0.5, 0) -- (21, 0);
    \draw[->] (0, -0.5) -- (0, 4);
    \draw[very thick] (0,0) to (3,3) to (4,2) to (5,3) to (6,2) to (7,3) to (10,0) to
    (11,1) to (12,0) to (14,2) to (15,1) to (17,3) to (20,0);
    \node[draw, rectangle, fill] at (7,3) {};
    \node[draw, rectangle, fill] at (11,1) {};
    \node[draw, rectangle, fill] at (17,3) {};
  \end{tikzpicture}
  $\quad\triangleq$
  \begin{tikzpicture}[thick, scale=0.65, baseline={([yshift=-0.3em]current bounding
      box.center)}]
    \node[draw, circle] {}
    child {node[draw, circle] {}
      child {node[draw, circle] {}
        child {node[draw, circle] {}}
        child {node[draw, circle] {}}
        child {node[draw, rectangle, fill] {}}
      }
    }
    child {node[draw, rectangle, fill] {}}
    child {node[draw, circle] {}
      child {node[draw, circle] {}}
      child {node[draw, circle] {}
        child {node[draw, rectangle, fill] {}}
      }
    };
  \end{tikzpicture}
  \caption{Bijection between Dyck paths with odd returns to zero and Catalan--Stanley
    trees. $\blacksquare$ marks all peaks before a descent to the $x$-axis and all
    rightmost leaves in the branches attached to the root, respectively.}
  \label{fig:bijection}
\end{figure}
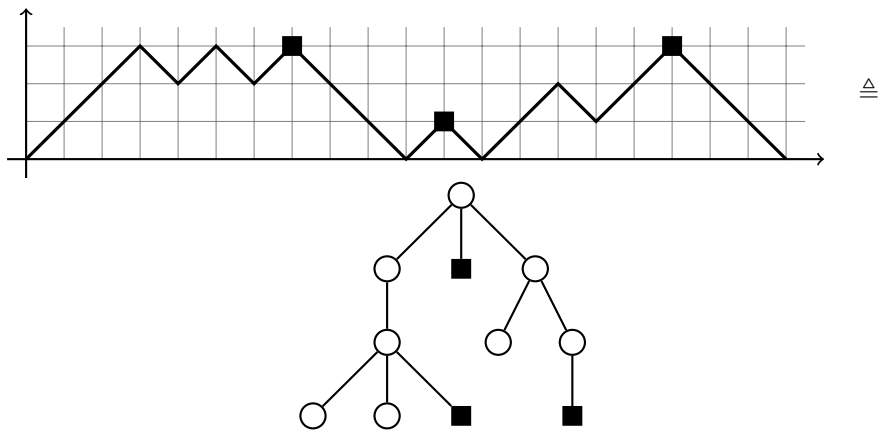

By the well-known glove bijection, this special class of Dyck paths corresponds to a special
class $\mathcal{S}$ of rooted plane trees, where the distance between the rightmost
node in all branches attached to the root and the root is odd. This bijection
is illustrated in Figure~\ref{fig:bijection}.

The trees in the combinatorial class $\mathcal{S}$ are the central object of study in this
paper.
\begin{definition}
  Let $\mathcal{S}$ be the combinatorial class of all rooted plane trees $\tau$, where the
  rightmost leaves in all branches attached to the root of $\tau$ have an odd
  distance to the root. In particular, $\innernode$ itself, i.e., the tree consisting of
  just the root belongs to $\mathcal{S}$ as well. We call the trees in $\mathcal{S}$
  \emph{Catalan--Stanley} trees.
\end{definition}

There are some recent approaches (see
\cite{Hackl-Heuberger-Kropf-Prodinger:ta:treereductions, Hackl-Heuberger-Prodinger:2018:register-reduction}) in which classical tree parameters like the register function
for binary trees are analyzed by, in a nutshell, finding a proper way to grow tree
families in a way that the parameter of interest corresponds to the age of the tree
within this (deterministic) growth process.

Following this idea, the aim of this paper is to define a ``natural'' growth process enabling us to
grow any Catalan--Stanley tree from $\innernode$, and then to analyze the corresponding
tree parameters.

In Section~\ref{sec:growing-trees} we define such a growth process and analyze some
properties of it. In particular, in Proposition~\ref{prop:family-expansion} we
characterize the family of trees that can be grown by applying a fixed number of
growth iterations to some given tree family. This is then used to derive generating functions
related to the parameters investigated in Sections~\ref{sec:age} and~\ref{sec:ancestors}.

Section~\ref{sec:age} contains an analysis of the age of Catalan--Stanley trees, asymptotic
expansions for the expected age among all trees of size $n$ and the corresponding variance
are given in Theorem~\ref{thm:age-analysis}.

Section~\ref{sec:ancestors} is devoted to the analysis of how fast trees of given size can
be grown by investigating the size of the $r$th ancestor tree compared to the size of the
original tree. This is characterized in Theorem~\ref{thm:ancestor-analysis}.

We use the open-source computer mathematics software
SageMath~\cite{SageMath:2017:7.6} with its included module for computing with asymptotic
expansions~\cite{Hackl-Heuberger-Krenn:2016:asy-sagemath} in order to carry out the
computationally heavy parts of this paper. The corresponding worksheet can be found at
\url{https://benjamin-hackl.at/publications/catalan-stanley/}.

\section{Growing Catalan--Stanley Trees}\label{sec:growing-trees}

We denote the combinatorial class of rooted plane trees with $\mathcal{T}$,
and the corresponding generating function enumerating these trees with respect
to their size by $T(z)$. For the sake of readability, we omit the argument of
$T(z) = T$ throughout this paper. By means of the symbolic
method~\cite[Chapter I]{Flajolet-Sedgewick:ta:analy}, the combinatorial
class $\mathcal{T}$ satisfies the construction $\mathcal{T} = \innernode \times
\operatorname{SEQ}(\mathcal{T})$. It translates into the functional equation
\begin{equation}\label{eq:simplify-T}
  T(z) = \frac{z}{1 - T(z)} \quad \iff \quad z + T(z)^{2} = T(z),
\end{equation}
which will be used throughout the paper. Additionally, it is easy to see by solving the
quadratic equation in~\eqref{eq:simplify-T} and choosing the correct branch of the
solution, we have the well-known formula $T(z) = \frac{1 - \sqrt{1 - 4z}}{2}$.

\begin{proposition}
  The generating function of the combinatorial class $\mathcal{S}$ of
  Catalan--Stanley trees, where $t$ marks all the rightmost nodes in the branches attached
  to the root of the tree and $z$ marks all other nodes, is given by
  \begin{equation}\label{eq:catalan-stanley-gf}
    S(z,t) = z + \frac{zt}{1 - t - T^{2}}.
  \end{equation}
  In particular, there is one Catalan--Stanley tree of size $1$ and $C_{n-2}$
  Catalan--Stanley trees of size $n$ for $n\geq 2$.
\end{proposition}
\begin{figure}[ht]
  \centering
  $\mathcal{S} = $\quad
  \begin{tikzpicture}[thick, scale=0.64, baseline={([yshift=-0.3em]current bounding
      box.center)}]
    \node[draw, circle] {};
  \end{tikzpicture}
  $\quad + \quad$
  \begin{tikzpicture}[thick, scale=0.64, baseline={([yshift=-0.3em]current bounding
      box.center)}, level distance=25mm, level 1/.style={sibling distance=45mm},]
    \node[draw, circle] {}
    child {node {$\operatorname{SEQ}\bigg(\tikz[scale=0.6, level distance=15mm]{\node {$\mathcal{T}$} child
          {node {$\mathcal{T}$}};}\bigg)$}
      child {node[draw, rectangle, fill] {}}
    }
    child {node {$\operatorname{SEQ}\bigg(\tikz[scale=0.6, level distance=15mm]{\node {$\mathcal{T}$} child
          {node {$\mathcal{T}$}};}\bigg)$}
      child {node[draw, rectangle, fill] {}}
    }
    child[gray, dashed] {node {\dots}}
    child {node {$\operatorname{SEQ}\bigg(\tikz[scale=0.6, level distance=15mm]{\node {$\mathcal{T}$} child
          {node {$\mathcal{T}$}};}\bigg)$}
      child {node[draw, rectangle, fill] {}}
    };
    \draw [thick, decoration={brace,mirror,raise=0.5cm}, decorate] (-7.5,-5) -- (7.5,-5)
    node [pos=0.5,anchor=north,yshift=-0.55cm] {$\geq 1$ branches}; 
  \end{tikzpicture}
  \caption{Symbolic specification of the combinatorial class $\mathcal{S}$ of
    Catalan--Stanley trees. Nodes represented by $\blacksquare$ are marked by the
    variable $t$, all other nodes are marked by $z$.}
  \label{fig:catalan-stanley-symbolic}
\end{figure}
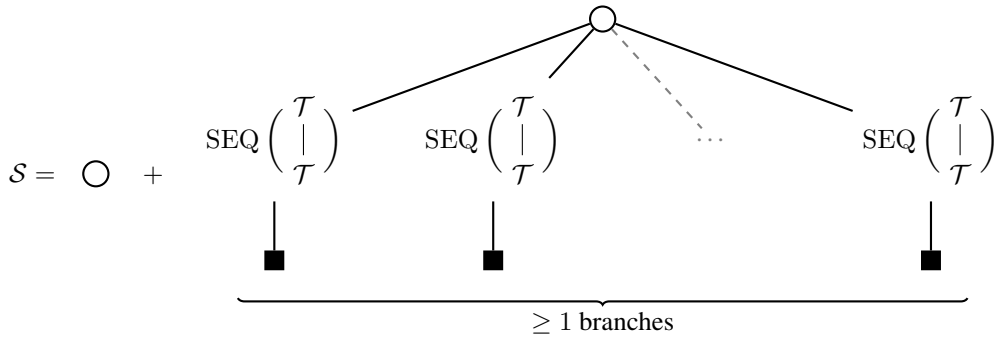
\begin{proof}
  By using the symbolic method~\cite[Chapter I]{Flajolet-Sedgewick:ta:analy}, the symbolic
  representation of $\mathcal{S}$ given in Figure~\ref{fig:catalan-stanley-symbolic}
  (which is based on the decomposition of the rightmost path in each subtree of the root
  into a sequence of pairs of rooted plane trees and the final rightmost leaf $\blacksquare$)
  translates into the functional equation
  \[ S(z,t) = z + \frac{z \frac{t}{1 - T^{2}}}{1 - \frac{t}{1 - T^{2}}},  \]
  which simplifies to the equation given in~\eqref{eq:catalan-stanley-gf}.

  In order to enumerate Catalan--Stanley trees with respect to
  their size, we consider $S(z,z)$, which simplifies to $z (T + 1)$ and thus proves
  the statement.
\end{proof}

We want to describe how to grow all Catalan--Stanley trees beginning from the tree that has
only one node, $\innernode$.

We consider the tree reduction $\rho \colon \mathcal{S} \to
\mathcal{S}$ that operates on a given Catalan--Stanley tree
$\tau$ (or just the root) as follows:

Start from all nodes that are represented by $t$, i.e.\ the rightmost
leaves in the branches attached to the root: if the node is a child of the root,
it is simply deleted. Otherwise we delete all subtrees of the grandparent of the
node and mark the resulting leaf, i.e., the former grandparent, with $t$.

\begin{figure}[ht]
  \centering
  \begin{tikzpicture}[thick, scale=0.65, baseline={([yshift=-0.3em]current bounding
      box.center)}]
    \node[draw, circle] {}
    child {
      node[draw, circle] {}
      child {
        node[draw, circle] {}
        child {
          node[draw, circle] {}
        }
      }
      child {
        node[draw, circle] {}
        child {
          node[draw, circle] {}
          child {node[draw, circle] {}
              child {node[draw, circle] {}}
              child {node[draw, circle] {}}
              child {node[draw, circle] {}}
          }
        }
        child{
          node[draw, rectangle, fill] {}
        }
      }
    }
    child {node[draw, rectangle, fill] {}}
    child {node[draw, circle] {}
      child{node[draw, circle] {}}
      child{node[draw, circle] {}
        child{node[draw, circle] {}
          child{node[draw, circle] {}
            child{node[draw, rectangle, fill] {}}
          }
        }
      }
    };
  \end{tikzpicture}
  $\quad\mapsto\quad$
  \begin{tikzpicture}[thick, scale=0.65, baseline={([yshift=-0.3em]current bounding
      box.center)}]
    \node[draw, circle] {}
    child {node[draw, rectangle, fill] {}}
    child {node[draw, circle] {}
      child {node[draw, circle] {}}
      child {node[draw, circle] {}
        child{node[draw, rectangle, fill] {}}
      }
    };
  \end{tikzpicture}
  $\quad\mapsto\quad$
  \begin{tikzpicture}[thick, scale=0.65, baseline={([yshift=-0.3em]current bounding
      box.center)}]
    \node[draw, circle] {} child {node[draw, rectangle, fill] {}};
  \end{tikzpicture}
  $\quad\mapsto\quad$
  \begin{tikzpicture}[thick, scale=0.65, baseline={([yshift=-0.3em]current bounding
      box.center)}]
    \node[draw, circle] {};
  \end{tikzpicture}
  \caption{Illustration of the reduction operator $\rho$, $\blacksquare$ marks the
    rightmost leaves in the branches attached to the root.}
  \label{fig:reduction-illustration}
\end{figure}
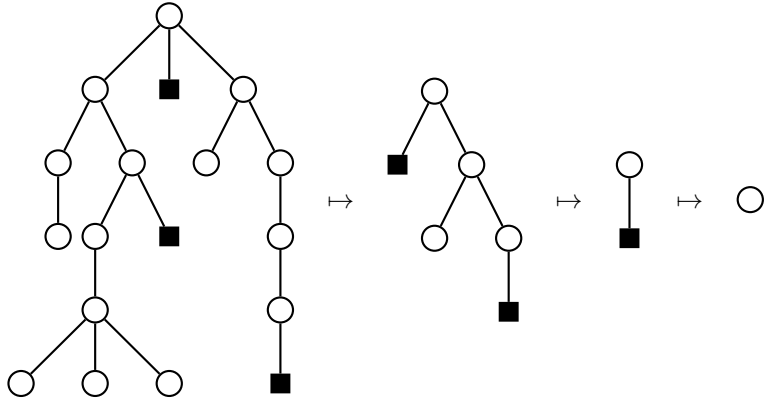

This tree reduction is illustrated in Figure~\ref{fig:reduction-illustration}. While the
reduction $\rho$ is certainly not injective as there are several trees with the same
reduction $\tau \in \mathcal{S}$, it is easy to construct a tree reducing to some given
$\tau\in \mathcal{S}$ by basically inserting chains of length $2$ before all rightmost
leaves in the branches attached to the root. This allows us to think of the operator
$\rho^{-1}$ mapping a given tree (or some family of trees) to the respective set of
preimages as a \emph{tree expansion operator}. In this context, we also want to define
the \emph{age} of a Catalan--Stanley tree.

\begin{definition}
  Let $\tau\in\mathcal{S}$ be a Catalan--Stanley tree. Then we define $\alpha(\tau)$, the
  \emph{age} of $\tau$, to be the number of expansions required to grow $\tau$ from the
  tree of size one, $\innernode$. In particular, we want
  \[ \alpha(\tau) = r \quad\iff\quad \tau\in (\rho^{-1})^{r}(\innernode) \text{ and }
    \tau\not\in (\rho^{-1})^{r-1}(\innernode)  \]
  for $r\in \Z_{\geq 1}$, and we set $\alpha(\innernode) = 0$.
\end{definition}
\begin{remark}
  Naturally, the concept of the age of a tree strongly depends on the underlying
  reduction procedure. In particular, for the reduction procedure considered in this
  article we have $\alpha(\tau) = r$ if and only if the maximum depth of the rightmost
  leaves in the branches attached to the root is $2r-1$.
\end{remark}

Before we delve into the analysis of the age of Catalan--Stanley trees, we need to be able
to translate the tree expansion given by $\rho^{-1}$ into a suitable form so that we can
actually use it in our analysis. The following proposition shows that $\rho^{-1}$ can be
expressed in the language of generating functions.

\begin{proposition}
  Let $\mathcal{F}\subseteq \mathcal{S}$ be a family of Catalan--Stanley trees
  with bivariate generating function $f(z,t)$, where $t$ marks rightmost
  leaves in the branches attached to the root and $z$ marks all other
  nodes. Then the generating function of $\rho^{-1}(\mathcal{F})$, the family of
  trees whose reduction is in $\mathcal{F}$, is given by
  \begin{equation}
    \label{eq:expansion-operator}
    \Phi(f(z,t)) = \frac{1}{1-t} f\Big(z, \frac{t}{1-t} T^{2}\Big).
  \end{equation}
\end{proposition}
\begin{proof}
  From a combinatorial point of view it is obvious that the operator $\Phi$ has to be
  linear, meaning that we can focus on determining all possible expansions of some tree
  represented by the monomial $z^{n}t^{k}$, i.e.\ a tree where the root has $k$ children
  (and thus $k$ different rightmost leaves in the branches attached to the root), and $n$ other nodes.

  In order to expand such a tree represented by $z^{n}t^{k}$ we begin by inserting a chain
  of length two before every rightmost leaf in order to ensure that the distance to the
  root is still odd. These newly inserted nodes can now be considered to be roots of some
  rooted plane trees, meaning that we actually insert two arbitrary rooted plane trees
  before every node represented by $t$. This corresponds to a factor of $t^{k}
  T^{2k}$.

  In addition to this operation, we are also allowed to add new children to the root,
  i.e.\ we can add sequences of nodes represented by $t$ before or after every child of
  the root. As observed above, the root has $k$ children and thus there are $k+1$
  positions where such a sequence can be attached. This corresponds to a
  factor of $(1-t)^{-(k+1)}$.

  Finally, we observe that nodes that are represented by $z$ are not expanded in any way,
  meaning that $z^{n}$ remains as it is.

  Putting everything together yields that
  \[ \Phi(z^{n}t^{k}) = \frac{1}{1-t} z^{n} \Big(\frac{t T^{2}}{1-t}\Big)^{k},  \]
  which, by linearity of $\Phi$, proves the statement.
\end{proof}

\begin{corollary}
  The generating function $S(z,t)$ satisfies the functional equation
  \[ \Phi(S(z,t)) = S(z,t).  \]
\end{corollary}
\begin{proof}
  This follows immediately from the fact that the reduction operator $\rho$ is surjective,
  as discussed above.
\end{proof}

Actually, in order to carry out a thorough analysis of this growth process for
Catalan--Stanley trees we need to have more information on the iterated application of the
expansion. In particular, we need a precise characterization of the family of
Catalan--Stanley trees that can be grown from some given tree family by expanding it a
fixed number of times.

\begin{proposition}\label{prop:family-expansion}
  Let $r\in \Z_{\geq 0}$ be fixed and $\mathcal{F}\subseteq \mathcal{S}$ be a family of
  Catalan--Stanley trees with bivariate generating function $f(z,t)$.
  Then the family of trees obtained by expanding the trees in $\mathcal{F}$ $r$ times is
  enumerated by the generating function
  \begin{equation}\label{eq:family-expansion}
    \Phi^{r}(f(z,t)) = \frac{1}{1 - t \frac{1 - T^{2r}}{1 - T^{2}}}
    f\Big(z, \frac{t T^{2r}}{1 - t \frac{1 - T^{2r}}{1 -
        T^{2}}}\Big).
  \end{equation}
\end{proposition}
\begin{proof}
  By linearity, it is sufficient to determine the generating function for the family of
  trees obtained by expanding some tree represented by $z^{n}t^{k}$. Consider the closely related multiplicative operator $\Psi$ with
  \[ \Psi(f(z,t)) = f\big(z, \frac{t}{1-t} T^{2}\big).  \]
  It is easy to see that we can write the $r$-fold application of $\Phi$ with
  the help of $\Psi$ as
  \[ \Phi^{r}(f(z,t)) = \Psi^{r}(f(z,t)) \prod_{j=0}^{r-1} \frac{1}{1 -
      \Psi^{j}(t)}. \]
  As $\Psi$ is multiplicative, we have
  \[ \Psi^{r}(z^{n}t^{k}) = \Psi^{r}(z)^{n} \Psi^{r}(t)^{k},  \]
  meaning that we only have to investigate the $r$-fold application of $\Psi$ to
  $z$ and to $t$.

  We immediately see that $\Psi^{r}(z) = z$, as $\Psi$ maps $z$ to $z$
  itself. For $\Psi^{r}(t)$, we can prove by induction that the relation
  \[ \Psi^{r}(t) = \frac{t T^{2r}}{1 - t \frac{1 - T^{2r}}{1 -
        T^{2}}}  \]
  holds for $r\geq 0$. Finally, observe that for $j\geq 1$ we have
  \begin{equation}\label{eq:observation}
    \Psi^{j}(t) = \frac{\Psi^{j-1}(t)}{1 - \Psi^{j-1}(t)} T^{2},
  \end{equation}
  and thus
  \[ \Psi^{r}(t) = \frac{\Psi^{r-1}(t)}{1 - \Psi^{r-1}(t)} T^{2} =
    \frac{\Psi^{r-2}(t)}{(1 - \Psi^{r-2}(t))(1 - \Psi^{r-1}(t))} T^{4} =
    \dots = \frac{t T^{2r}}{\prod_{j=0}^{r-1} (1 - \Psi^{j}(t))}  \]
  by iteratively using~\eqref{eq:observation} in the numerator. With our explicit
  formula for $\Psi^{r}(t)$ from above this yields
  \[ \prod_{j=0}^{r-1} (1 - \Psi^{j}(t)) = 1 - t \frac{1 - T^{2r}}{1 -
      T^{2}} \]
  for $r\geq 1$. Putting everything together we obtain
  \[ \Phi^{r}(z^{n}t^{k}) = \frac{1}{1 - t \frac{1 - T^{2r}}{1 - T^{2}}}
    z^{n} \Psi^{r}(t)^{k}, \]
  which proves~\eqref{eq:family-expansion} by linearity of $\Phi^{r}$.
\end{proof}

From this characterization we immediately obtain the generating functions for all the
classes of objects we will investigate in the following sections.

\begin{corollary}\label{cor:leq-eq-gf}
  Let $r\in\Z_{\geq 0}$. The generating function $F_{r}^{\leq}(z,t)$ enumerating
  Catalan--Stanley trees of age less than or equal to $r$ where $t$ marks the rightmost
  leaves in the branches attached to the root and $z$ marks all other nodes is given by
  \begin{equation}\label{eq:leq-gf}
      F_{r}^{\leq}(z,t) = \frac{z}{1 - t \frac{1 - T^{2r}}{1 - T^{2}}}.
  \end{equation}
\end{corollary}
\begin{proof}
  As we defined $\rho(\innernode) = \innernode$ we have $\innernode \in
  \rho^{-1}(\innernode)$, which implies $F_{r}^{\leq}(z,t)$ is given by $\Phi^{r}(z)$.
\end{proof}

\begin{corollary}\label{cor:size-change-gf}
  Let $r \geq 0$. Then the generating function $G_{r}(z,v)$ enumerating
  Catalan--Stanley trees where $z$ marks the tree size and $v$ marks the size of
  the $r$-fold reduced tree, is given by
  \begin{equation}\label{eq:size-change-gf}
    G_{r}(z,v) = \Phi^{r}(S(zv,tv))|_{t=z} = \frac{1}{1 - z\frac{1  -T^{2r}}{1 -
    T^{2}}} S\bigg(zv, \frac{z T^{2r}}{1 - z \frac{1 - T^{2r}}{1 - T^{2}}}v\bigg).  \end{equation}
\end{corollary}
\begin{proof}
  Observe that the generating function $S(zv,tv)$ enumerates Catalan--Stanley trees with
  respect to the number of rightmost leaves in the branches attached to the root (marked by
  $t$), the number of other nodes (marked by $z$), and the size of the tree (marked by
  $v$). Applying the operator $\Phi^{r}$ to this generating function thus yields a
  generating function where $v$ still marks the size of the tree, $t$ and $z$ however
  enumerate the number of rightmost leaves in the branches attached to the root and all
  other nodes of the $r$-fold expanded tree, respectively. After setting $t=z$, we obtain
  a generating function where $v$ marks the size of the original tree and $z$ the size of
  the $r$-fold expanded tree---which is equivalent to the formulation in the corollary. 
\end{proof}

\section{The Age of Catalan--Stanley Trees}\label{sec:age}

In this section we want to give a proper analysis of the parameter $\alpha$ defined in the
previous section. Formally, we do this by considering the random variable $D_{n}$
modeling the age of a tree of size $n$, where all Catalan--Stanley trees of size $n$ are
equally likely.

\begin{remark}
  It is noteworthy that in~\cite{Hackl-Heuberger-Prodinger:2018:register-reduction} it was shown
  that the well-known register function of a binary tree can also be obtained as the
  number of times some reduction can be applied to the binary tree until it
  degenerates. The age of a Catalan--Stanley tree can thus be seen as a ``register
  function''-type parameter as well.
\end{remark}

First of all, we are interested in the minimum and maximum age a tree of size $n$ can
have.

\begin{proposition}\label{prop:age-bounds}
  Let $n\in \Z_{\geq 2}$. Then the bounds
  \begin{equation}\label{eq:age-bounds}
    1 \leq D_{n} \leq \Big\lfloor\frac{n}{2}\Big\rfloor
  \end{equation}
  hold and are sharp, i.e.\ there are trees $\tau$, $\tau'\in\mathcal{S}$ of size $n\geq
  2$ such that $D_{n}(\tau) = 1$ and $D_{n}(\tau') = \lfloor n/2\rfloor$ hold. The only tree of size $1$ is
  $\innernode$, and it satisfies $D_{1}(\innernode) = 0$.
\end{proposition}
\begin{proof}
  Note that only $\innernode$, the tree of size $1$ has age $0$, therefore the lower bound
  is certainly valid for trees of size $n \geq 2$. This lower bound is sharp, as the tree
  with $n-1$ children attached to the root is a Catalan--Stanley tree and has age $1$.

  For the upper bound, first observe that given a tree of size $n\geq 3$ the reduction
  $\rho$ always removes at least $2$ nodes from the tree. If the tree is of size $2$, then
  $\rho$ only removes one node. Given an arbitrary Catalan--Stanley tree $\tau$ of age $r$
  and size $n$, this means that
  \[ 1 = \abs{\innernode} = \abs{\rho^{r}(\tau)} \leq \abs{\tau} - 2\cdot(r-1) - 1 = n -
    2r + 1,  \]
  where $\abs{\tau}$ denotes the size of the tree $\tau$. This yields $r \leq n/2$,
  and as $r$ is known to be an integer we may take the floor of the number on the
  right-hand side of the inequality. This proves that the upper bound
  in~\eqref{eq:age-bounds} is valid.

  The upper bound is sharp because we can construct appropriate families of trees precisely
  reaching the upper bound: for even $n$, the chain of size $n$ is a Catalan--Stanley tree
  of age $n/2$. For odd $n = 2m+1$ we consider the chain of size $2m$ and attach one
  node to the root of it. The resulting tree is a Catalan--Stanley tree of age $m = \lfloor
  n/2\rfloor$, and thus proves that the bound is sharp.
\end{proof}

By investigating the generating functions obtained from Corollary~\ref{cor:leq-eq-gf} we
can characterize the limiting distribution of the age of Catalan--Stanley trees when the
size $n$ tends to $\infty$.

\begin{theorem}\label{thm:age-analysis}
  Consider $n\to\infty$. Then the age of a (uniformly random) Catalan--Stanley tree of size
  $n$ behaves according to a discrete limiting distribution where
  \begin{multline}\label{eq:prob-age}
    \P(D_{n} = r) = \Big(\frac{4(4^{r}(3r-1) + 1)}{(4^{r} + 2)^{2}} - \frac{4(4^{r+1}(3r+2) +
      1)}{(4^{r+1} + 2)^{2}}\Big)\\
    - \Big(
    \frac{6\cdot 64^{r}(2r^{3} - 5r^{2} + 4r - 1) - 6\cdot 16^{r}(16r^{3} - 24r^{2} +10r -
      1) + 24\cdot 4^{r} (2r^{3} - r^{2})}{(4^{r}+2)^{4}} \\
    - \frac{6\cdot 64^{r+1} (2r^{3} + r^{2}) - 6\cdot 16^{r+1} (16r^{3} + 24r^{2} + 10r + 1)
      + 24\cdot 4^{r+1} (2r^{3} + 5r^{2} + 4r + 1)}{(4^{r+1} + 2)^{4}}
    \Big) n^{-1}\\ + O\Big(\frac{r^{5}}{3^{r}} n^{-2}\Big)
  \end{multline}
  for $r \in \Z_{\geq 1}$, and the $O$-term holds uniformly in $r$. Additionally, by setting
  \begin{align*}
    c_{0} & = \sum_{r\geq 1} \frac{4^{r+1} (3r-1) + 4}{(4^{r} + 2)^{2}}\\
          & = 2.7182536428679528526648361928219367344585435680344\ldots,\\
    c_{1} & = - \sum_{r\geq 1} \frac{6\cdot 64^{r} (2r^{3} - 5r^{2} + 4r - 1) - 6\cdot
            16^{r}(16r^{3} - 24r^{2} + 10r - 1) + 24\cdot 4^{r} (2r-1)
            r^{2}}{(4^{r}+2)^{4}}\\
          & = -4.2220971510158840823821873477600478080816411210406\ldots,  \\
    c_{2} & = \sum_{r\geq 1} (2r-1)\frac{4^{r+1} (3r-1) + 4}{(4^{r} + 2)^{2}} - c_{0}^{2}\\
          & = 0.91845604214374797357797147814019496503688953933967\ldots,\\
    c_{3} & = \begin{multlined}[t]
            - \sum_{r\geq 1} \frac{(2r-1)}{(4^{r}+2)^{4}} \big(6\cdot 64^{r} (2r^{3} - 5r^{2} + 4r - 1) - 6\cdot
            16^{r}(16r^{3} - 24r^{2} + 10r - 1)\\
            + 24\cdot 4^{r} (2r-1) r^{2}\big) - 2c_{0}c_{1}
            \end{multlined}\\
          & = -9.1621753200836274996912436568310268988536534594942\ldots,\\
  \end{align*}
  the expected age and the corresponding variance are given by the asymptotic expansions
  \begin{equation}\label{eq:age-expectation}
    \E D_{n} = c_{0} + c_{1} n^{-1} + O(n^{-2}),
  \end{equation}
  \begin{equation}\label{eq:age-variance}
    \V D_{n} = c_{2} + c_{3} n^{-1} + O(n^{-2}).
  \end{equation}
\end{theorem}
\begin{proof}
  For the sake of convenience we set $F_{r}^{\leq}(z) := F_{r}^{\leq}(z,z)$, where
  $F_{r}^{\leq}(z,t)$ is given in~\eqref{eq:leq-gf}. This univariate generating
  function now enumerates Catalan--Stanley trees of age $\leq r$ with respect to the tree
  size.

  We begin by observing that $F_{r}^{\geq}(z)$, the generating function enumerating
  Catalan--Stanley trees of age $\geq r$ with respect to the tree size is given by
  \begin{equation}\label{eq:geq-gf}
    F_{r}^{\geq}(z) = S(z,z) - F_{r-1}^{\leq}(z) = z(1 + T) - \frac{z}{1 - z\frac{1 -
        T^{2r-2}}{1 - T^{2}}} = z(1 + T) \frac{T^{2r-1}}{1 + T^{2r-1}},
  \end{equation}
  where the last equation follows after some elementary manipulations and by using~\eqref{eq:simplify-T}.

  Now let $f_{n,r} := [z^{n}]F_{r}^{\geq}(z)$ denote the number of Catalan--Stanley trees
  of size $n$ and age $\geq r$. As we consider all Catalan--Stanley trees of size $n$ to be
  equally likely, we find
  \[ \P(D_{n} = r) = \P(D_{n} \geq r) - \P(D_{n} \geq r+1) = \frac{f_{n,r} -
      f_{n,r+1}}{C_{n-2}}. \]
  We use singularity analysis (see~\cite{Flajolet-Odlyzko:1990:singul}
  and~\cite[Chapter VI]{Flajolet-Sedgewick:ta:analy}) in order to obtain an asymptotic
  expansion for $f_{n,r}$. To do so, we first observe that $z=1/4$ is the dominant
  singularity of $T$ and thus also of $F_{r}^{\geq}(z)$. We then
  consider $z$ to be in some $\Delta$-domain at
  $1/4$ (see~\cite[Definition VI.1]{Flajolet-Sedgewick:ta:analy}). The task of expanding
  $F_{r}^{\geq}(z)$ for $z\to 1/4$ now largely consists of handling the term
  $\frac{T^{2r-1}}{1 + T^{2r-1}}$. Observe that we can write
  \begin{align*}
    \frac{T^{2r-1}}{1 + T^{2r-1}} & = \frac{1}{1 + T^{1-2r}}
                                    = \frac{1}{1 + 2^{2r-1} (1 - \sqrt{1 - 4z}\,)^{1-2r}},\\
                                  & = \frac{1}{(1 + 2^{2r-1})\big(1 + \frac{2^{2r-1}}{1 +
                                    2^{2r-1}} \sum_{j\geq 1} \binom{2r+j-2}{j} (1-4z)^{j/2}\big)}
  \end{align*}
  which results in
  \begin{multline*}
    \frac{T^{2r-1}}{1 + T^{2r-1}} = \frac{2}{4^{r} + 2} - \frac{2\cdot 4^{r}
      (2r-1)}{(4^{r} + 2)^{2}} (1 - 4z)^{1/2} \\
    + \frac{2\cdot 4^{r}\, (4^{r}(r-1) - 2r)(2r-1)}
    {(4^{r} + 2)^{3}} (1 - 4z)\\
    - \frac{2\cdot 4^{r}(16^{r} (2r^{2} - 5r + 3) - 4^{r+2}(r^{2} - r) + 8r^{2} + 4r)
      (2r-1)}{3 (4^{r}+2)^{4}} (1 - 4z)^{3/2}
    + O\Bigl(\frac{r^{4}}{3^{r}}(1 -
    4z)^{2}\Bigr),
  \end{multline*}
  where the $O$-term holds uniformly in $r$. Multiplying this expansion with the expansion
  of $z (1+T)$ yields the expansion
  \begin{multline*}
    F_{r}^{\geq}(z) = \frac{3}{4(4^{r} + 2)} - \frac{4^{r}(3r-1) + 1}{2 (4^{r} + 2)^{2}}
    (1 - 4z)^{1/2} \\ + \frac{16^{r} (6r^{2} - 7r - 1) - 2\cdot 4^{r} (6 r^{2} - 5r + 7) -
      12}{4 (4^{r} + 2)^{3}} (1 - 4z)\\
    - \frac{64^{r}(2r^{3} - 5r^{2} + r) - 2\cdot 16^{r} (8r^{3} - 12r^{2} + 11r -2)
    + 4^{r+1} (2r^{3} - r^{2} - 3r) - 4}{2 (4^{r} + 2)^{4}} (1 - 4z)^{3/2}
    \\+ O\Bigl(\frac{r^{4}}{3^{r}}(1-4z)^{2}\Bigr).
  \end{multline*}
  By means of singularity analysis we extract the $n$th coefficient and find
  \begin{multline*}
    f_{n,r} = \frac{4^{r} (3r-1) + 1}{4\sqrt{\pi}\, (4^{r} + 2)^{2} } 4^{n} n^{-3/2}
     - \Big(\frac{3\cdot 64^{r} (8r^{3} - 20r^{2} + r + 1)}{32\sqrt{\pi}\, (4^{r} + 2)^{4}} \\
     - \frac{3\cdot 16^{r}(64r^{3} - 96r^{2} +
       100r - 19) - 12\cdot 4^{r} (8r^{3} - 4r^{2} - 15r) + 60}{32\sqrt{\pi}\, (4^{r} +
       2)^{4}}\Big) 4^{n} n^{-5/2}
     \\ + O\Big(\frac{r^{5}}{3^{r}} 4^{n} n^{-7/2}\Big).
  \end{multline*}
  Computing the difference $f_{n,r} - f_{n,r+1}$ and dividing by the Catalan number
  $C_{n-2}$ then yields the expression for $\P(D_{n} = r)$ given in~\eqref{eq:prob-age}.

  The expected value can then be computed with the help of the well-known formula
  \[ \E D_{n} = \sum_{r\geq 1} \P(D_{n} \geq r), \]
  which proves~\eqref{eq:age-expectation}. Finally, the variance can be obtained from $\V
  D_{n} = \E(D_{n}^{2}) - (\E D_{n})^{2}$, where
  \[ \E(D_{n}^{2}) = \sum_{r\geq 1} r^{2} \P(D_{n} = r) = \sum_{r\geq 1} (2r - 1) \P(D_{n}
    \geq r), \]
  which proves~\eqref{eq:age-variance}.
\end{proof}

In addition to the asymptotic expansions given in Theorem~\ref{thm:age-analysis} we can
also determine an exact formula for the expected value $\E D_{n}$. The key tools in this
context are Cauchy's integral formula as well as the substitution $z =
\frac{u}{(1+u)^{2}}$.

\begin{proposition}\label{prop:age-exact}
  Let $n\in \Z_{\geq 2}$. The expected age of the Catalan--Stanley trees of size $n$ is
  given by
  \begin{equation}\label{eq:age-exact}
    \E D_{n} = \frac{1}{C_{n-2}} \sum_{k\geq 1} (-1)^{k+1} \sigma_{0}^{\mathrm{odd}}(k)
    \Bigg(\binom{2n- 4 - k}{n-3} + \binom{2n-4-k}{n-2} - 2\binom{2n-4-k}{n-1}\Bigg),
  \end{equation}
  where $\sigma_{0}^{\mathrm{odd}}(k)$ denotes the number of odd divisors of $k$.
\end{proposition}
\begin{proof}
  We begin by explicitly extracting the coefficient $[z^{n}] F_{r}^{\geq}(z)$. The
  expected value can then be obtained by summation over $r$ and division by $C_{n-2}$.

  With the help of the substitution $z = \frac{u}{(1+u)^{2}}$ we can bring
  $F_{r}^{\geq}(z)$ into the more suitable form
  \[  F_{r}^{\geq}(z) = \frac{(1+2u) u^{2r}}{(1+u)^{3} (u^{2r-1} + (1+u)^{2r-1})}. \]
  We extract the coefficient of $z^{n}$ now by means of Cauchy's integral formula. Let
  $\gamma$ be a small contour winding around the origin once. Then we have
  \begin{align}
    [z^{n}] F_{r}^{\geq}(z)
    & = \frac{1}{2\pi i} \oint_{\gamma} \frac{F_{r}^{\geq}(z)}{z^{n+1}}~dz
    = \frac{1}{2\pi i} \oint_{\tilde\gamma} \frac{(1+u)^{2n+2}}{u^{n+1}} \frac{(1+2u)
      u^{2r}}{(1+u)^{3} (u^{2r-1} + (1+u)^{2r-1})} \frac{1-u}{(1+u)^{3}}~du \notag\\
    & = [u^{n-2r}] (1+2u)(1-u) (1+u)^{2n-2r-3} \frac{1}{1 + (\frac{u}{1+u})^{2r-1}} \notag\\
    & = [u^{n-2r}] (1 + u - 2u^{2}) \sum_{j\geq 1} (-1)^{j-1} u^{(2r-1)(j-1)}
      (1+u)^{2n-4-j(2r-1)}\notag\\
    & = \begin{multlined}[t]
      \sum_{j\geq 1} (-1)^{j-1} \Bigg(\binom{2n-4 - j(2r-1)}{n-3} +
      \binom{2n-4 - j(2r-1)}{n-2} \\ - 2 \binom{2n-4 - j(2r-1)}{n-1}\Bigg),
      \end{multlined} \label{eq:age-exact:r}
  \end{align}
  where $\tilde\gamma$, the integration contour of the second integral, is the
  transformation of $\gamma$ under the transformation $z = u/(1+u)^{2}$ and is also a
  small contour winding around the origin once.

  Now consider the auxiliary sum
  \[ \vartheta(k) := \sum_{\substack{j,r\geq 1 \\ j(2r-1) = k}} (-1)^{j-1}.  \]
  It is easy to see by distinguishing between even and odd $k$ that with the help of
  $\sigma_{0}^{\mathrm{odd}}(k)$,
  $\vartheta(k)$ can be written as $\vartheta(k) = (-1)^{k-1}
  \sigma_{0}^{\mathrm{odd}}(k)$.

  Summing the expression from~\eqref{eq:age-exact:r} over $r\geq 1$, simplifying the
  resulting double sum by means of the auxiliary sum $\vartheta$, and finally dividing by
  $C_{n-2}$ then proves~\eqref{eq:age-exact}.
\end{proof}

\section{Analysis of Ancestors}\label{sec:ancestors}

In this section we focus on characterizing the effect of the (repeatedly applied)
reduction $\rho$ on a random Catalan--Stanley tree of size $n$. We are particularly
interested in studying the size of the reduced tree. In the light of the fact that all
Catalan--Stanley trees can be grown from $\innernode$ by means of the growth process
induced by $\rho$, we can think of the $r$th reduction of some tree $\tau$ as the $r$th
ancestor of $\tau$.

In order to formally conduct this analysis, we consider the random variable $X_{n,r}$
modeling the size of the $r$th ancestor of some tree of size $n$, where all Catalan--Stanley
trees of size $n$ are equally likely.

Similar to our approach in Proposition~\ref{prop:age-bounds} we can determine precise
bounds for $X_{n,r}$ as well.

\begin{proposition}\label{prop:ancestor-bounds}
  Let $n\in \Z_{\geq 2}$ and $r\in \Z_{\geq 1}$. Then the bounds
  \begin{equation}\label{eq:ancestor-bounds}
    1 \leq X_{n,r} \leq n - 2(r-1) - 1
  \end{equation}
  hold for $r \leq \lfloor n/2 \rfloor$ and are sharp, i.e.\ there are trees $\tau$,
  $\tau'\in \mathcal{S}$ of size $n\geq 2$ such that $X_{n,r}(\tau) = 1$ and
  $X_{n,r}(\tau') = n - 2(r-1) - 1$. For $r > \lfloor n/2 \rfloor$ the variable $X_{n,r}$ is
  deterministic with $X_{n,r} = 1$.
\end{proposition}
\begin{proof}
  Assume that $r \leq \lfloor n/2\rfloor$. The lower bound is obvious as trees cannot
  reduce further than to $\innernode$, and as the first ancestor of the tree with $n-1$
  children attached to the root already is $\innernode$ the lower bound is valid and
  sharp.

  For the upper bound we follow the same argumentation as in the proof of
  Proposition~\ref{prop:age-bounds} to arrive at
  \[ 1 \leq \abs{\rho^{r}(\tau)} \leq \abs{\tau} - 2(r-1) - 1  = n - 2r + 1  \]
  for some Catalan--Stanley tree of size $n$, which proves that the upper bound is
  valid. Any tree $\tau$ of size $n$ having the chain of length $2$ as its $(r-1)$th ancestor
  satisfies $X_{n,r}(\tau) = n - 2(r-1) - 1$ and thus proves that the upper bound is
  sharp. This proves~\eqref{eq:ancestor-bounds}.

  In the case of $r > \lfloor n/2\rfloor$ we observe that as the $\lfloor n/2\rfloor$th
  ancestor of any Catalan--Stanley tree of size $n$ already is certain to be $\innernode$
  by Proposition~\ref{prop:age-bounds}, the $r$th ancestor is $\innernode$ as well.
\end{proof}

With the generating function $G_{r}(z,v)$ enumerating Catalan--Stanley trees with respect
to their size (marked by $n$) and the size of their $r$th ancestor (marked by $v$) from
Corollary~\ref{cor:size-change-gf} we can write the probability generating function of
$X_{n,r}$ as
\[ \E v^{X_{n,r}} = \frac{1}{C_{n-2}} [z^{n}]G_{r}(z,v). \]
This allows us to extract parameters like the expected size of the $r$th ancestor and the
corresponding variance.

\begin{theorem}\label{thm:ancestor-analysis}
  Let $r\in \Z_{\geq 0}$ be fixed and consider $n\to\infty$. Then the expected value and
  the variance of the random variable $X_{n,r}$ modeling the size of the $r$th ancestor
  of a (uniformly random) Catalan--Stanley tree of size $n$ are given by the asymptotic
  expansions
  \begin{equation}\label{eq:reduced-size:expectation}
    \E X_{n,r} = \frac{1}{4^{r}} n + \frac{2\cdot 4^{r} - 2r^{2} + r - 2}
    {2\cdot 4^{r}} +
    \frac{(2r+1)(2r-1)(r-3)r}{2\cdot 4^{r+1}} n^{-1} + O(n^{-3/2}),
  \end{equation}
  \begin{multline}\label{eq:reduced-size:variance}
    \V X_{n,r} = \frac{(2^{r} +1)(2^{r} - 1)}{16^{r}} n^{2} - \frac{\sqrt{\pi}
      (4^{r}(3r+1) - 1)}{3\cdot 16^{r}} n^{3/2} \\ + \frac{18\cdot 4^{r} r^{2} + 3\cdot 4^{r}
    r  - 38\cdot 4^{r} + 36r^{2} - 42r + 38 }{18\cdot 16^{r}}n \\ + \frac{5\sqrt{\pi}\,
    (4^{r}(3r+1) - 1)}{8\cdot 16^{r}} n^{1/2} + O(1).
  \end{multline}
\end{theorem}
\begin{proof}
  The strategy behind this proof is to determine the first and second factorial moment of
  $X_{n,r}$ by extracting the coefficient of $z^{n}$ in the derivatives
  $\frac{\partial^{d}}{\partial v^{d}} G_{r}(z,v)|_{v=1}$ for $d\in\{1,2\}$ and normalizing the
  result by dividing by $C_{n-2}$.

  We begin with the expected value. With the help of SageMath~\cite{SageMath:2017:7.6} we
  find for $z\to 1/4$
  \begin{multline*}
    \frac{\partial}{\partial v} G_{r}(z,v)|_{v=1} = \frac{1}{4^{r+2}} (1 - 4z)^{-1/2} +
    \frac{3\cdot 4^{r} -r}{2\cdot 4^{r+1}} - \frac{2\cdot 4^{r} - 2r^{2} + r + 2}{4^{r+2}}
    (1 - 4z)^{1/2} \\
    - \frac{9\cdot 4^{r} + 2r^{3} - 3r^{2} - 5r}{6\cdot 4^{r+1}} (1-4z)
    + O((1-4z)^{3/2}),
  \end{multline*}
  where the $O$-constant depends implicitly on $r$. Extracting the coefficient of $z^{n}$
  and dividing by $C_{n-2}$ yields the expansion given
  in~\eqref{eq:reduced-size:expectation}.

  Following the same approach for the second derivative yields the expansion
  \begin{multline*}
    \frac{\partial^{2}}{\partial v^{2}} G_{r}(z,v)|_{v=1} =
    \frac{1}{2\cdot 4^{r+2}} (1 - 4z)^{-3/2}
    - \frac{4^{r} (3r+1) - 1}{3\cdot 16^{r+1}} (1 - 4z)^{-1}
    \\+ \frac{4^{r}(18r^{2} + 3r +7) - 24r + 2}{18\cdot 16^{r+1}} (1 - 4z)^{-1/2}
    + O(1),
  \end{multline*}
  such that after applying singularity analysis and division by $C_{n-2}$ we obtain the
  expansion
  \begin{multline*}
    \E X_{n,r}^{\underline{2}} = \frac{1}{4^{r}} n^{2}
    - \frac{\sqrt{\pi}\, (4^{r} (3r+1) - 1)}{3\cdot 16^{r}} n^{3/2}
    + \frac{4^{r} (18r^{2} + 3r - 20) - 24r + 2}{18\cdot 16^{r}} n
    \\ + \frac{5\sqrt{\pi}\, (4^{r} (3r+1) - 1)}{8\cdot 16^{r}} n^{1/2}
    + O(1)
  \end{multline*}
  for the second factorial moment $\E X_{n,r}^{\underline{2}}$. Applying the well-known
  formula
  \[ \V X_{n,r} = \E X_{n,r}^{\underline{2}} + \E X_{n,r} - (\E X_{n,r})^{2}  \]
  then leads to the asymptotic expansion for the variance given
  in~\eqref{eq:reduced-size:variance} and thus proves the statement.
\end{proof}

Besides the asymptotic expansion given in Theorem~\ref{thm:ancestor-analysis}, we are also
interested in finding an exact formula for the expected value $\E X_{n,r}$. We can do so
by means of Cauchy's integral formula.

\begin{proposition}\label{prop:ancestor-exact}
  Let $n$, $r\in \Z_{\geq 1}$. Then the expected size of the $r$th ancestor of a random
  Catalan--Stanley tree of size $n$ is given by
  \begin{equation}\label{eq:ancestor:expectation-exact}
    \E X_{n,r} = \frac{1}{C_{n-2}} \binom{2n-2r-4}{n-2} + 1.
  \end{equation}
\end{proposition}
\begin{proof}
  We rewrite the derivative $g(z) := \frac{\partial}{\partial v} G_{r}(z,v)|_{v=1}$ into a more suitable
  form which makes it easier to extract the coefficients. To do so, we use the
  substitution $z = u/(1+u)^{2}$ again, allowing us to express the derivative as
  \[ g(z) = \frac{u^{2r+2}}{(1-u)(1+u)^{2r+3}} + \frac{(1+2u)u}{(1+u)^{3}}. \]
  Note that as $T = \frac{u}{1+u}$, the summand $\frac{(1+2u)u}{(1+u)^{3}}$ actually
  represents $z (1 + T)$, implying that the coefficient of $z^{n}$ in this
  summand is given by $C_{n-2}$.
  Now let $\gamma$ be a small contour winding around the origin once, so that with
  Cauchy's integral formula we obtain
  \begin{align*}
    [z^{n}]g(z)
    & = \frac{1}{2\pi i} \oint_{\gamma} \frac{g(z)}{z^{n+1}}~dz =
      \frac{1}{2\pi i} \oint_{\tilde\gamma} \frac{(1+u)^{2n+2}}{u^{n+1}}
      \frac{u^{2r+2}}{(1-u)(1+u)^{2r+3}} \frac{1-u}{(1+u)^{3}}~du + C_{n-2}\\
    & = [u^{n - 2r - 2}] (1+u)^{2n-2r-4} + C_{n-2} = \binom{2n-2r-4}{n-2r-2} + C_{n-2},
  \end{align*}
  where $\tilde\gamma$ is the image of $\gamma$ under the transformation (and is still a
  small contour winding around the origin once). Dividing by $C_{n-2}$ then proves~\eqref{eq:ancestor:expectation-exact}.
\end{proof}

\newpage
\bibliographystyle{amsplainurl}
\bibliography{bib/cheub}

\end{document}